\documentclass{article}

\usepackage{amssymb,latexsym,amsmath,amsthm,amsfonts,graphics}
\usepackage{graphicx}
\graphicspath{ {Figures/} }

\usepackage{caption}
\usepackage{subcaption}
\usepackage[rightcaption]{sidecap}

\usepackage{color}
\usepackage{lineno}
\usepackage{multirow}
\usepackage{epstopdf}
\usepackage{rotating}
\usepackage{cite}

\usepackage[a4paper, total={6.8in, 9in}]{geometry}
\usepackage{hyperref}
\usepackage{tikz}

\newtheorem{thm}{Theorem}[section]
\newtheorem{cor}{Corollary}[section]
\newtheorem{lem}{Lemma}[section]

\newtheorem{dfn}{Definition}[section]
\newtheorem{ex}{Example}[section]

\newtheorem{rem}{Remark}[section]
\newcommand{\Z}{\mathbb{Z}}

\makeatletter
\def\ps@pprintTitle{%
 \let\@oddhead\@empty
 \let\@evenhead\@empty
 \def\@oddfoot{\centerline{\thepage}}%
 \let\@evenfoot\@oddfoot}
\makeatother

\begin{document}

\begin{center}
{\bf {\Large Some Necessary and Sufficient Conditions for Diophantine Graphs}}\\
\end{center}

\begin{center}
{ \bf M. A. Seoud*$^3$, \ A. Elsonbaty*$^2$, \ A. Nasr*$^1$, \ M. Anwar*$^4$}
\vspace{3mm}\\
 *Department of Mathematics, Faculty of Science, Ain Shams University, 11566, Abbassia, Cairo, Egypt.
\vspace{3mm}\\
e-mails: $^1$ \ \href{mailto:amr_fatouh@sci.asu.edu.eg}{\url{amr_fatouh@sci.asu.edu.eg}}, $^2$ \ \href{mailto:ahmedelsonbaty@sci.asu.edu.eg}{\url{ahmedelsonbaty@sci.asu.edu.eg}},\\
 \hspace{0.9cm}$^3$ \ \href{mailto:m.a.seoud@sci.asu.edu.eg}{\url{m.a.seoud@sci.asu.edu.eg}},\hspace{0.2cm} $^4$ \ \href{mailto:mohamedanwar@sci.asu.edu.eg}{\url{mohamedanwar@sci.asu.edu.eg}},
\end{center}

\begin{center}
  MSC code: 05A10, 05C07, 05C78, 11A05, 11A25, 11B75, 11D04, 11D88.
\end{center}

%
%
%


\begin{abstract}
A linear Diophantine equation $ax+by=n$ is solvable if and only if $\gcd(a,b)$ divides $n$. A graph $G$ of order $n$ is called Diophantine if there exists a labeling function $f$ of vertices such that $\gcd(f(u),f(v))$ divides $n$ for every two adjacent vertices $u,v$ in $G$. In this work, maximal Diophantine graphs on $n$ vertices, $D_n$, are defined, studied and generalized. The independence number, the number of vertices with full degree and the clique number of $D_n$ are computed. Each of these quantities is the basis of a necessary condition for the existence of such a labeling.
\end{abstract}

\begin{flushleft}
\textbf{Keywords}: Diophantine graph, Maximal Diophantine graph, labeling isomorphism, $\gamma$-labeled graph.
\end{flushleft}

\section{Introduction}

\hspace{0.5cm} Assuming that a graph $G=(V, E)$ is a finite simple undirected graph with $|V|$ vertices and $|E|$ edges, where $V=V(G)$ is the vertex set, $E=E(G)$ is the edge set, $|V|$ is called the order of the graph $G$ and $|E|$ is called the size of the graph $G$. In general, $|X|$ denotes the cardinality of a set $X$. $\delta(G)$ denotes the minimum degree of the vertices in a graph $G$. A set of vertices $S$ of a graph $G$ is said to be an independent set or a free set if for all $u,v\in S$, $u,v$ are nonadjacent in $G$. The independence number, denoted by $\alpha(G)$, is the maximum order of an independent set of vertices of a graph $G$. The operation of adding an edge $e=uv$ to a graph $G$ joining the vertices $u,v$ yields a new graph with the same vertex set $V(G)$ and edge set $E(G)\cup\{uv\}$, which is denoted $G+\{uv\}$. The operation of deleting an edge $e=uv$ from a graph $G$ removes only that edge, the resulting graph is denoted $G-\{uv\}$. A spanning subgraph of a graph $G$ is a subgraph of $G$ obtained by deleting edges only, adding edges to a graph $G$ yields a spanning supergraph of $G$. The join of two graphs $G$ and $H$ is denoted by $G+H$, it has the following vertex set $V(G+H)= V(G)\cup V(H)$ and edge set $E(G+H)=E(G)\cup E(H)\cup\{uv: u\in V(G) \ \mbox{and} \ v\in V(H)\}$. 
$K_n,\overline{K_n}$ and $C_n$ denote the complete graph, the null graph and the cycle graph of order $n$ respectively. 
We follow terminology and notations in graph theory as in A. Bickle \cite{Bickle}, J. L. Gross; J. Yellen; P. Zhang \cite{G-Y-Z}, F. Harary \cite{Harary} and K. H. Rosen \cite{Rosen2}.

The concept of prime labeling was introduced by R. Entringer and was discussed in a paper by A. Tout \cite{Tout}. A graph $G$ is called a prime graph if there exists a bijective map $f:V\rightarrow \{1, 2, \dots, n\}$ such that for all $uv\in E$, $(f(u),f(v))=1$.  Some authors investigated algorithms for prime labeling in  \cite{sonbaty} and necessary and sufficient conditions are studied in \cite{Seoud1}, \cite{Seoud-Y}. The notion of Diophantine labeling is an extension of that of prime labeling. In this paper, we give a brief summary of some definitions and some results pertaining to Diophantine graphs. A generalization encompassing prime graphs, Diophantine graphs and another type of graph labeling is introduced and discussed. In maximal Diophantine graphs, an arithmetic function is established to calculate the number of vertices with full degree and the order of the maximal clique or the maximal complete subgraph, the independence number is computed and necessary and sufficient conditions are provided with these bounds. Moreover, an explicit formula for a vertex with minimum degree and minimum label is proved. Furthermore, a new perspective on degree sequences for establishing necessary conditions is presented.  Relevant definitions and notations from number theory are mentioned. We follow the basic definitions and notations of number theory as in T. M. Apostol \cite{Apostol} and D. Burton \cite{Burton}.



This manuscript is structured as follows. Section 2  provides some results of $\gamma$-labelings. Section 3 is partitioned into three subsections, each presents some results related to maximal Diophantine graphs. Subsection 3.1 discusses some basic bounds and necessary and sufficient conditions for maximal Diophantine graphs. Subsection 3.2 and 3.3 provided some necessary conditions  and explore properties of the minimum degree and the degree sequence in maximal Diophantine graphs. Section 4 includes some examples of non-Diophantine graphs to explain the relation among these necessary conditions.

\begin{dfn}\label{dfn2}\cite{Nasr}
  Let $G$ be a graph with $n$ vertices. The graph $G$ is called a Diophantine graph if there exists a bijective map $f:V\rightarrow \{1, 2, \dots, n\}$ such that for all $uv\in E$, $(f(u),f(v))\mid n$. Such a map $f$ is called a Diophantine labeling of $G$. A maximal Diophantine graph with $n$ vertices, denoted by $(D_n,f)$, is a Diophantine graph such that adding any new edge yields a non-Diophantine graph. If there is no ambiguity, we drop $f$ from $(D_n,f)$ and write it simply $D_n$.
\end{dfn}
Clearly, if a graph $G$ is Diophantine, then $|E(G)|\leq|E(D_n)|$. A formula that computes the number of edges of $D_n$ can be found in \cite{Nasr}.  Some maximal Diophantine graphs are given in the next example.

\begin{ex}
The following three graphs are examples of maximal Diophantine graphs.
\begin{figure*}[h!]
\centering
\begin{subfigure}{0.3\textwidth}
 \centering
  \begin{tikzpicture}
  [scale=.6,auto=center,every node/.style={circle,fill=blue!20}]
  \node (v9)   at (0,4)           {$9$};
  \node (v1)   at (3,2.5)         {$1$};
  \node (v7)   at (3.7,0)         {$7$};
  \node (v5)   at (-3,2.5)        {$5$};
  \node (v3)   at (-3.7,0)        {$3$};

  \node (v2)[circle,fill=red!20]  at (-3,-2.5)    {$2$};
  \node (v4)[circle,fill=red!20]  at (-1,-3)      {$4$};
  \node (v6)[circle,fill=red!20]  at (1,-3)       {$6$};
  \node (v8)[circle,fill=red!20]  at (3,-2.5)     {$8$};
  \draw (v1) -- (v2);
  \draw (v1) -- (v3);
  \draw (v1) -- (v4);
  \draw (v1) -- (v5);
  \draw (v1) -- (v6);
  \draw (v1) -- (v7);
  \draw (v1) -- (v8);
  \draw (v1) -- (v9);

  \draw (v3) -- (v2);
  \draw (v3) -- (v4);
  \draw (v3) -- (v5);
  \draw (v3) -- (v6);
  \draw (v3) -- (v7);
  \draw (v3) -- (v8);
  \draw (v3) -- (v9);

  \draw (v5) -- (v2);
  \draw (v5) -- (v4);
  \draw (v5) -- (v6);
  \draw (v5) -- (v7);
  \draw (v5) -- (v8);
  \draw (v5) -- (v9);

  \draw (v7) -- (v2);
  \draw (v7) -- (v4);
  \draw (v7) -- (v6);
  \draw (v7) -- (v8);
  \draw (v7) -- (v9);

  \draw (v9) -- (v2);
  \draw (v9) -- (v4);
  \draw (v9) -- (v6);
  \draw (v9) -- (v8);
  \end{tikzpicture}\caption{Graph $D_9$}
 \end{subfigure}
~~~
\begin{subfigure}{0.3\textwidth}
 \centering
  \begin{tikzpicture}
  [scale=.6,auto=center,every node/.style={circle,fill=blue!20}]
  \node (v4)   at (3.5,0)         {$4$};

  \node (v1)   at (3.7,2)           {$1$};
  \node (v2)   at (2.5,4)           {$2$};
  \node (v10)  at (0,4.9)           {$10$};
  \node (v7)   at (-2.5,4)          {$7$};
  \node (v5)   at (-3.7,2)          {$5$};

  \node (v8)   at (-3.5,0)        {$8$};

  \node (v3)[circle,fill=red!20]  at (0,-2.5)     {$3$};
  \node (v6)[circle,fill=red!20]  at (-2,-2)      {$6$};
  \node (v9)[circle,fill=red!20]  at (2,-2)       {$9$};
  \draw (v1) -- (v2);
  \draw (v1) -- (v3);
  \draw (v1) -- (v4);
  \draw (v1) -- (v5);
  \draw (v1) -- (v6);
  \draw (v1) -- (v7);
  \draw (v1) -- (v8);
  \draw (v1) -- (v9);
  \draw (v1) -- (v10);

  \draw (v5) -- (v2);
  \draw (v5) -- (v3);
  \draw (v5) -- (v4);
  \draw (v5) -- (v6);
  \draw (v5) -- (v7);
  \draw (v5) -- (v8);
  \draw (v5) -- (v9);
  \draw (v5) -- (v10);

  \draw (v7) -- (v2);
  \draw (v7) -- (v3);
  \draw (v7) -- (v4);
  \draw (v7) -- (v6);
  \draw (v7) -- (v8);
  \draw (v7) -- (v9);
  \draw (v7) -- (v10);

  \draw (v2) -- (v3);
  \draw (v2) -- (v4);
  \draw (v2) -- (v6);
  \draw (v2) -- (v8);
  \draw (v2) -- (v9);
  \draw (v2) -- (v10);

  \draw (v10) -- (v3);
  \draw (v10) -- (v4);
  \draw (v10) -- (v6);
  \draw (v10) -- (v8);
  \draw (v10) -- (v9);

  \draw (v4) -- (v3);
  \draw (v4) -- (v6);
  \draw (v4) -- (v9);

  \draw (v8) -- (v3);
  \draw (v8) -- (v6);
  \draw (v8) -- (v9);
 \end{tikzpicture}\caption{Graph $D_{10}$}
 \end{subfigure}
~~
 \begin{subfigure}{0.25\textwidth}
 \centering
  \begin{tikzpicture}
  [scale=.6,auto=center,every node/.style={circle,fill=blue!20}]
  \node (v9)   at (3.7,0)         {$9$};
  \node (v1)   at (3,2.5)         {$1$};
  \node (v11)  at (1.5,4)         {$11$};
  \node (v7)   at (-1.5,4)        {$7$};
  \node (v5)   at (-3,2.5)        {$5$};
  \node (v3)   at (-3.7,0)        {$3$};

  \node (v2)[circle,fill=red!20]  at (-3,-2.5)    {$2$};
  \node (v4)[circle,fill=red!20]  at (-1.5,-3)    {$4$};
  \node (v6)[circle,fill=red!20]  at (0,-3.5)     {$6$};
  \node (v8)[circle,fill=red!20]  at (1.5,-3)     {$8$};
  \node (v10)[circle,fill=red!20] at (3,-2.5)     {$10$};
  \draw (v1) -- (v2);
  \draw (v1) -- (v3);
  \draw (v1) -- (v4);
  \draw (v1) -- (v5);
  \draw (v1) -- (v6);
  \draw (v1) -- (v7);
  \draw (v1) -- (v8);
  \draw (v1) -- (v9);
  \draw (v1) -- (v10);
  \draw (v1) -- (v11);

  \draw (v11) -- (v2);
  \draw (v11) -- (v3);
  \draw (v11) -- (v4);
  \draw (v11) -- (v5);
  \draw (v11) -- (v6);
  \draw (v11) -- (v7);
  \draw (v11) -- (v8);
  \draw (v11) -- (v9);
  \draw (v11) -- (v10);

  \draw (v7) -- (v2);
  \draw (v7) -- (v3);
  \draw (v7) -- (v4);
  \draw (v7) -- (v5);
  \draw (v7) -- (v6);
  \draw (v7) -- (v8);
  \draw (v7) -- (v9);
  \draw (v7) -- (v10);

  \draw (v5) -- (v2);
  \draw (v5) -- (v3);
  \draw (v5) -- (v4);
  \draw (v5) -- (v6);
  \draw (v5) -- (v8);
  \draw (v5) -- (v9);

  \draw (v3) -- (v2);
  \draw (v3) -- (v4);
  \draw (v3) -- (v8);
  \draw (v3) -- (v10);

  \draw (v9) -- (v2);
  \draw (v9) -- (v4);
  \draw (v9) -- (v8);
  \draw (v9) -- (v10);
  \end{tikzpicture} \caption{Graph $D_{11}$}
 \end{subfigure}\caption{Some maximal Diophantine graphs $D_9$, $D_{10}$ and $D_{11}$}\label{figure0}
\end{figure*}
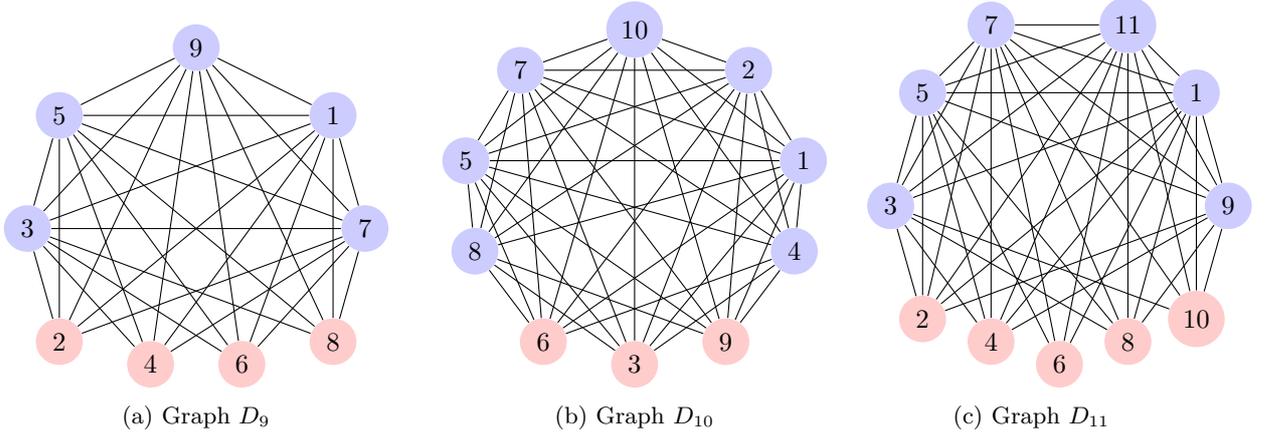
\end{ex}

\begin{dfn}\cite{Nasr}
  For a given an integer $n\in \Z^+$ and a prime $p\in \mathbb{P}$, the successor of the $p$-adic valuation is denoted by $\acute{v}_p(n):=v_p(n)+1$, where $v_p(n)$ is the $p$-adic valuation, $ \Z^+$ is set of positive integers and $\mathbb{P}$ is the set of prime numbers. The number $p^{\acute{v}_p(n)}$ is called the critical prime power number with respect to $p,n$.
\end{dfn}

In the rest of this paper, the following arithmetic functions $\pi,\omega$ and $\tau$ will be used, (see \cite{Apostol}, \cite{Burton}): Let $n\in \Z^+$.
\begin{equation*}
     \pi(n):=\big|\{p\in\mathbb{P}: 2\leq p\leq n\}\big|, \quad \omega(n):=\big|\{p\in\mathbb{P}: p\mid n, \ 2\leq p\leq n\}\big|,
      \quad\tau(n):=\big|\{d\in \Z^+ : d\mid n\}\big|.
\end{equation*}

\begin{lem}\label{lem1}\cite{Nasr}
  Suppose $D_n$ is a maximal Diophantine graph of order $n$. For every $u,v\in V(D_n)$, $uv\notin E(D_n)$ if and only if there exists $p\in\mathbb{P}$ such that 
  $$f(u), f(v)\in M_{p^{\acute{v}_{p}(n)}}:=\left\{kp^{\acute{v}_{p}(n)}: \ k=1,2,\dots,\left\lfloor\frac{n}{p^{\acute{v}_{p}(n)}}\right\rfloor\right\}.$$
\end{lem}

\begin{thm}\label{lem2}\cite{Nasr}
   Suppose $D_n$ is a maximal Diophantine graph of order $n$. For every $u\in V(D_n)$,
   $$\deg(u)=n-1\quad\mbox{if and only if}\quad f(u)\mid n\quad\mbox{\textbf{or}}\quad \frac{n}{2}<f(u)=p^{\acute{v}_p(n)}<n,$$
   where $p\in\mathbb{P}$ and the exclusive \textbf{or} will be typed in bold while the inclusive or is as usual.
\end{thm}


The reduced label $f^*(u)$ of a vertex $u$ in a labeled graph $G$ with $n$ vertices is defined as $f^*(u):=\frac{f(u)}{(f(u), n)}.$

\begin{lem}\label{lem3}\cite{Nasr}
   Suppose $D_n$ is a maximal Diophantine graph of order $n$ and $u,v\in V(D_n)$. If $f(u)\mid f(v)$, then $N(u)\supseteq N(v)$, where $N(s)$ defines the neighborhood of $s$ as the set of all vertices in $D_n$ that join the vertex $s$.
\end{lem}

\begin{thm}\label{thm_eq-deq2}\cite{Nasr}
 Suppose $D_n$ is a maximal Diophantine graph of order $n$. Let $u,v\in V(D_n)$ such that $f(u)\mid f(v)$, $f(v)$ is not a prime power number and $f^*(u)>1$. If $\deg(u)=\deg(v)$, then $f^*(u),f^*(v)$ have the same prime factors.
\end{thm}

\begin{cor}\label{cor1}\cite{Nasr}
  Suppose $D_n$ is a maximal Diophantine graph of order $n$ and $u,v\in V(D_n)$ such that $f(v)=tf(u)$  for some $t\geq1$. If $t\mid n$ and $(t, f(u))=1$, then $\deg(u)=\deg(v)$.
\end{cor}


\section{$\gamma$-Labelings of Graphs }

\hspace{0.cm}The following definition is a generalization of Definition \ref{dfn2}.
\begin{dfn}\label{dfn3}
  Let $G$ be a graph with $n$ vertices. The graph $G$ is called an $\gamma$-labeled graph if there exists a bijective map $f:V\rightarrow \{x_1, x_2, \dots, x_n\}$ such that $f(u),f(v)$ satisfy some conditions, where $\{x_1, x_2, \dots, x_n\}$ is any set of $n$ elements. Such a map $f$ is called an $\gamma$-labeling. A maximal $\gamma$-labeled graph with $n$ vertices, denoted by $(\Gamma_n,f)$, is a $\gamma$-labeled graph in which for all $uv\notin E(\Gamma_n)$, $\Gamma_n+\{uv\}$ is not a $\gamma$-labeled graph.
\end{dfn}

The reader should not be confused  the notion of $\gamma$-labeling as provided in Definition \ref{dfn3} with the concept of $\alpha$-valuation that presented in the seminal work of A. Rosa \cite{Rosa}.


\begin{dfn}\cite{S-C-L}
  Let $(G_1,f_1),(G_2,f_2)$ be two labeled graphs, where $f_1:V(G_1)\rightarrow \{x_1, x_2, \dots, x_n\}$ and $f_2:V(G_2)\rightarrow \{x_1, x_2, \dots, x_n\}$ are two bijective maps. The labeled graphs $(G_1,f_1),(G_2,f_2)$ are said to be labeling isomorphic, denoted by $(G_1,f_1)\cong_l (G_2,f_2)$, if there exists a bijective map $\varphi:V(G_1)\rightarrow V(G_2)$ such that for all $u,v\in V(G_1)$, $uv\in E(G_1)$ if and only if $\varphi(u)\varphi(v)\in E(G_2)$ and $f_1(u)=\big(f_2\circ\varphi\big)(u).$
\end{dfn}

\begin{thm}\label{thm-equivalance}
  A maximal $\gamma$-labeled graph $\Gamma_n$ is unique up to labeling isomorphism.
\end{thm}
\begin{proof}
  Suppose $(\Gamma_n,f_1)$ and $(\acute{\Gamma}_n,f_2)$ are two maximal $\gamma$-labeled graphs of order $n$, where the two maps 
  $$f_1:V(\Gamma_n)\rightarrow \{x_1, x_2, \dots, x_n\}\quad \mbox{and}\quad f_2:V(\acute{\Gamma}_n)\rightarrow \{x_1, x_2, \dots, x_n\}$$
  are $\gamma$-labelings of $\Gamma_n$ and $\acute{\Gamma}_n$ satisfying certain conditions, say condition $C$. Define a map 
  $$\varphi:V(\Gamma_n)\rightarrow V(\acute{\Gamma}_n)\quad \mbox{by}\quad \varphi(u)=f_2^{-1}(f_1(u)).$$ 
  Therefore, $\varphi$ is one to one (for let $u,v\in V(\Gamma_n)$, $\varphi(u)=\varphi(v)$. Then we obtain $f_2^{-1}(f_1(u))=f_2^{-1}(f_1(v))$; accordingly, $f_1(u)=f_1(v)$. Consequently, $u=v$), $\varphi$ is onto (since $\varphi$ is one to one and $|V(\Gamma_n)|=|V(\acute{\Gamma}_n)|=n$), $\varphi$ is preserving the adjacency and non-adjacency of $\Gamma_n$ and $\acute{\Gamma}_n$ (for the reason that let $u,v\in V(\Gamma_n)$ such that $uv\in E(\Gamma_n)$. Then we have the two labels $f_1(u),f_1(v)$ satisfy $C$. Since, $f_1(u)=f_2(\varphi(u))$ and $f_1(v)=f_2(\varphi(v))$ (see Figure \ref{fig.}), we get $f_2(\varphi(u)),f_2(\varphi(v))$ satisfy $C$. Consequently, $\varphi(u)\varphi(v)\in E(\acute{\Gamma}_n)$ and the converse is similar) and let $u\in V(\Gamma_n)$, $\varphi(u)=f_2^{-1}(f_1(u))$. Therefore, $f_1(u)=f_2(\varphi(u))=(f_2\circ\varphi)(u)$. Hence, the two graphs $(\Gamma_n,f_1)$ and $(\acute{\Gamma}_n,f_2)$ are labeling isomorphic.
\end{proof}

\begin{figure*}[h!]
\centering
  \begin{tikzpicture}
   [scale=.8,auto=center]
  \node (v)  at (0,1.33) {$\equiv$};
  \node (v1) at (0,0)  {$\{x_1, x_2, \dots, x_n\}$};
  \node (v2) at (-2,2) {$V(\Gamma_n)$};
  \node (v3) at (2,2)  {$V(\acute{\Gamma}_n)$};
  \path[->] (v2)edge [align=left, below] node  {$f_1$} (v1);
  \path[->] (v3)edge [align=left, below] node  {$f_2$} (v1);
  \path[->] (v2)edge [align=left, above] node  {$\varphi$} (v3);
  \end{tikzpicture}
  \caption{$(\Gamma_n,f_1)\cong_l (\acute{\Gamma}_n,f_2)$}\label{fig.}
\end{figure*}

\begin{cor}\label{thm-equivalance1}
   The graphs $D_n$ are unique up to labeling isomorphism.
\end{cor}

\begin{thm}
  Suppose $G$ is a graph with order $n$ and $\Gamma_n$ is the maximal $\gamma$-labeled graph with order $n$. $G$ is an $\gamma$-labeled graph if and only if $G$ is labeling isomorphic to a spanning subgraph of $\Gamma_n$.
\end{thm}
\begin{proof}
  Suppose $\Gamma_n$ is the maximal $\gamma$-labeled graph with order $n$ and a graph $G$ is a $\gamma$-labeled graph with order $n$. Then there exists $f:V(G)\rightarrow \{x_1, x_2, \dots, x_n\}$ is a bijective map such that $f(u),f(v)$ satisfy certain conditions, say condition $C$ and define 
  $$T:=\{uv:uv\notin E(G) \ \mbox{and} \ f(u),f(v) \ \mbox{satisfy} \ C\}.$$
  Consequently, the spanning supergraph $G+T$ of $G$ is a $\gamma$-labeled graph of order $n$ and the set $E(G)\cup T$ is set of all edges such that $f(u),f(v)$ satisfy $C$. Let $\acute{u}\acute{v}\notin E(G)\cup T$. Then we have that the two labels $f(\acute{u}),f(\acute{v})$ do not satisfy  $C$. Therefore, the spanning supergraph $G+(T\cup\{\acute{u}\acute{v}\})$ of $G$ is not a $\gamma$-labeled graph with a $\gamma$-labeling satisfy $C$. Consequently, $G+T$ is the maximal $\gamma$-labeled graph of order $n$. Thus, using Theorem \ref{thm-equivalance}, we have that $G+T$ is labeling isomorphic to $\Gamma_n$. Hence, the graph $G$ is labeling isomorphic to a spanning subgraph of the maximal $\gamma$-labeled graph $\Gamma_n$.\\

  Conversely, suppose $\Gamma_n$ is the maximal $\gamma$-labeled graph with order $n$ and a graph $G$ is labeling isomorphic to a spanning subgraph of the maximal $\gamma$-labeled graph $\Gamma_n$. Let $T$ be the set of deleted edges of $\Gamma_n$ such that the graph $G$ is labeling isomorphic to $\Gamma_n-T$. Then we have
  $$|V(G)|=|V(\Gamma_n-T)|=|V(\Gamma_n)| \quad  \mbox{and} \quad V(\Gamma_n)=V(\Gamma_n-T).$$
  Therefore, using the same $\gamma$-labeling of $\Gamma_n$, we have $\Gamma_n-T$ is a $\gamma$-labeled graph. Since the graph $G$ is labeling isomorphic to $\Gamma_n-T$, hence the graph $G$ is a $\gamma$-labeled graph.
\end{proof}

\begin{cor}\label{spanning-thm}
  A graph $G$ of order $n$ is  Diophantine if and only if $G$ is labeling isomorphic to a spanning subgraph of $D_n$.
\end{cor}

%

\section{Basic Bounds of the Maximal Diophantine Graphs $D_n$}

\subsection{Some Necessary and Sufficient Conditions for $D_n$ }

\hspace{0.5cm} In what follows, let $(D_n,f)$ denote the maximal Diophantine graph of order $n$, with Diophantine labeling $f$ and $F(G)$ denote the number of full degree vertices of a graph $G$. The next two theorems present two different methods that compute the quantity $F(D_n)$. 
\begin{thm}\label{fulldegree2}
   If $p_i^{\acute{v}_{p_i}(n)}<\frac{n}{2}$, $i=1, 2, \dots, r$, then the number of full degree vertices in $D_n$ is given by
\begin{equation*}
  F(D_n) =n-\sum_{1\leq i\leq r}\left\lfloor\frac{n}{p_i^{\acute{v}_{p_i}(n)}}\right\rfloor
           +\sum_{1\leq i<j\leq r}\left\lfloor\frac{n}{p_i^{\acute{v}_{p_i}(n)}p_j^{\acute{v}_{p_j}(n)}}\right\rfloor
           -\dots +(-1)^{r}\left\lfloor\frac{n}{\prod\limits_{1\leq i\leq r}p_i^{\acute{v}_{p_i}(n)}}\right\rfloor,
\end{equation*}
  where $p_1, p_2, \dots, p_r$ are distinct prime numbers.
\end{thm}
The proof of Theorem \ref{fulldegree2} is straightforward by applying Lemma \ref{lem1}, Theorem \ref{lem2} and the inclusion-exclusion principle (see \cite{Rosen2}). 
For a very large $n\in \Z^+$, the above formula does not provide efficient upper and lower bounds for the quantity $F(D_n)$. There is an alternative approach to determine the quantity $F(D_n)$ by using the following arithmetic function
$$\gamma_x(n):=\left|\left\{p^{\acute{v}_p(n)}: p\mid n, \ x<p^{\acute{v}_p(n)}<n, \ p\in\mathbb{P}\right\}\right|,$$
where $n\in \Z^+$ and a positive real number $x<n$. This function is utilized for computing not only the number of vertices with full degree in $D_n$ but also the order of the maximal clique of $D_n$ as follows in Theorems \ref{fulldegree}, \ref{complete_subgraph}. Obviously, for every $n\in \Z^+$, $\gamma_1(n)\leq\omega(n)$, for every $p\in\mathbb{P}$, $k\in \Z^+$ and a positive real number $x<n$, $\gamma_x\left(p^k\right)=0$ and also, for every $n,m\in\Z^+$ with $m<n$, $\gamma_m(n)=\gamma_1(n)-\gamma_1(m)$. 

\begin{thm} \label{fulldegree}
  The number of vertices with full degree in $D_n$ is given by
\begin{equation*}
F(D_n)=\tau(n) + \pi(n-1)-\pi\left(\frac{n}{2}\right) + \gamma_{\frac{n}{2}}(n).
\end{equation*}
In particular, if $n$ is a prime number, we have 
$$F(D_n)=\pi(n)-\pi\left(\frac{n}{2}\right) +1.$$
\end{thm}
\begin{proof}
  Let $D_n$ be the maximal Diophantine graph with order $n$. Define the following three sets
\begin{equation*}
     S_1:=\{d\in \Z^+ : d\mid n\}, \quad
     S_2:=\left\{p\in\mathbb{P}: \frac{n}{2} < p < n\right\}, \quad
     S_3:=\left\{ p^{\acute{v}_p(n)} : p\mid n, \ \frac{n}{2}< p^{\acute{v}_p(n)} < n, \ p\in\mathbb{P} \right\}.
\end{equation*}
Consequently, using Theorem \ref{lem2}, one can see that $ S_1\cup S_2\cup S_3$ is the set of labels of the full degree vertices in $D_n.$  Clearly, $S_1,S_2$ and $S_3$ are mutually disjoint sets and

  $$|S_1|=\tau(n),\quad |S_2|=\pi(n-1)-\pi\left(\frac{n}{2}\right)\quad \mbox{and}\quad |S_3|=\gamma_{\frac{n}{2}}(n),$$
and hence 
$$F(D_n)= \tau(n) + \pi(n-1)-\pi\left(\frac{n}{2}\right)  + \gamma_{\frac{n}{2}}(n).$$
In case of $n$ is a prime number, we have $F(D_n)= \pi(n)-\pi\left(\frac{n}{2}\right)+1$.
\end{proof}


\begin{cor}\label{corVI2}
  Let $G$ be a graph with order $n$. If the graph $G$ is Diophantine, then $F(G)\leq F(D_n)$.
\end{cor}


The clique number, denoted by $Cl(G)$, is the order of the maximal clique of a graph $G$. Although $\omega(G)$ is the standard notation of the clique number, we have chosen $Cl(G)$ in this study to prevent confusion with the arithmetic function $\omega(n)$. The following theorem gives the order of the maximal clique in $D_n$.

\begin{thm}\label{complete_subgraph} 
The clique number of $D_n$ is given by  
$$Cl(D_n)= \tau(n) + \pi(n) - \omega(n) + \gamma_1(n).$$ 
In particular, if $n$ is a prime number, we have
$$Cl(D_n)=\pi(n)+1.$$
\end{thm}
\begin{proof}
  Let $D_n$ be the maximal Diophantine graph with order $n$. Define the following three sets
\begin{equation*}
     S_1:=\{d\in \Z^+ : d\mid n\}, \quad
     S_2:=\{p\in\mathbb{P}: p\nmid n, \ 1 < p < n\}, \quad
     S_3:=\left\{p^{\acute{v}_p(n)}: p\mid n, \ 1<p^{\acute{v}_p(n)}<n, \ p\in\mathbb{P}\right\}.
\end{equation*}
  Therefore, any two vertices in $V(D_n)$ that is labeled by integers from the set $S_1\cup S_2\cup S_3$ are adjacent, since for any two distinct labels  $\ell_1,\ell_2$, we have
   \begin{equation*}
    \begin{cases}
       (\ell_1, \ell_2)=1,    & \mbox{if} \ \ell_1, \ell_2\in S_2\cup S_3\\ 
       &\\
       (\ell_1, \ell_2)\mid n, & \mbox{if} \ \ell_1\in S_1. \\
    \end{cases}
   \end{equation*}
 Consequently, one can see that $ S_1\cup S_2\cup S_3$ is the set of labels of vertices that are in the maximal clique of $D_n.$
 Suppose contrary that $u\in V(D_n)$ is a vertex $u$ of the maximal clique in $D_n$ such that $f(u)\notin S_1\cup S_2\cup S_3.$ Then we have $f(u)\nmid n$. Therefore, there exists a prime number $p_0$ such that $p_0^{\acute{v}_{p_0}(n)}\mid f(u)$; otherwise, for every a prime number $p$, $p^{\acute{v}_p(n)}\nmid f(u)$, so we get $v_p(f(u))<\acute{v}_p(n)=v_p(n)+1$. Consequently, $v_p(f(u))\leq v_p(n)$ which is a contradiction of $f(u)\nmid n$. Let $\ell=p_0^{\acute{v}_{p_0}(n)}$ be a certain label. Then we have $\ell\in S_2\cup S_3$, $\ell\mid f(u)$ and $\ell\neq f(u)$. So, $(f(u),\ell)=\ell\nmid n,$ which  contradicts the completeness of the maximal clique in $D_n$. Therefore, the set $S_1\cup S_2\cup S_3$ has all labels of vertices in the maximal clique of $D_n$. Obviously, $S_1, S_2$ and $S_3$ are mutually disjoint sets and
 $$|S_1|=\tau(n),\quad |S_2|=\pi(n)-\omega(n)\quad \mbox{and}\quad |S_3|=\gamma_1(n),$$
 we obtain
 $$Cl(D_n)=\tau(n) + \pi(n) - \omega(n) + \gamma_1(n).$$
 If $n$ is a prime number, then $Cl(D_n)=\pi(n)+1.$
\end{proof}

\begin{cor} \label{corVI3}
  Let $G$ be a graph with order $n$. If the graph $G$ is Diophantine, then $Cl(G)\leq Cl(D_n)$.
\end{cor}

\begin{rem} Let $D_n$ be the maximal Diophantine graph of order $n$. Then
\begin{itemize}
  \item[1.] $|E(D_n)|\geq\frac{1}{2}Cl(D_n)\big(Cl(D_n)-1\big)\geq \frac{1}{2}F(D_n)\big(F(D_n)-1\big),$
  \item[2.] if $D_n$ is not a complete graph, then $F(D_n)\leq\delta(D_n)$,
  \item[3.] for every $n\in \Z^+$, $F(D_n)\leq Cl(D_n)\leq n$.
\end{itemize}
\end{rem}

\begin{lem}
   For every a prime number $p\leq\frac{n}{2}$, $p\mid n$ and $p^{\acute{v}_p(n)}>\frac{n}{2}$ if and only if $D_n$ is a complete graph.
\end{lem}
\begin{proof}
  Assume $p\leq\frac{n}{2}$ is prime number such that $p\mid n$ and $p^{\acute{v}_p(n)}>\frac{n}{2}$. Suppose contrary that the maximal Diophantine graph $D_n$ is not a complete graph. Then there exist $u,v\in V(D_n)$ such that $uv\notin E(D_n)$. Therefore, using lemma \ref{lem1}, there exists a prime number $p$ such that $f(u),f(v)\in M_{p^{\acute{v}_p(n)}}$. Let $f(u)=tp^{\acute{v}_p(n)}$ and $f(v)=s p^{\acute{v}_p(n)}$ for some $t,s\geq1$ and $t<s$. Then, $p^{\acute{v}_p(n)}<\frac{n}{s}\leq\frac{n}{2},$
  this contradicts the assumption. Hence, $D_n$ is a complete graph.\\

  Conversely, let $D_n$ be a complete graph and consider contrary that there exists a prime number $p\leq\frac{n}{2}$ such that $p\nmid n$ or $p^{\acute{v}_p(n)}<\frac{n}{2}$, otherwise, if $p^{\acute{v}_p(n)}=\frac{n}{2}$, then $p^{\acute{v}_p(n)}\mid n$ that is a contradiction. Then we have the following two cases. In case of $p\leq\frac{n}{2}$ and $p\nmid n$, we obtain $2p<n$. Then we get $(p, 2p)=p\nmid n$. Therefore,  $F(D_n)<n$. In the other case of $p^{\acute{v}_p(n)}<\frac{n}{2}$, we have $(p^{\acute{v}_p(n)}, 2p^{\acute{v}_p(n)})= p^{\acute{v}_p(n)}\nmid n$. Therefore, $F(D_n)<n$. Consequently, from the two cases, $D_n$ is not a complete graph, this contradicts the hypothesis.
\end{proof}

\begin{thm}
  The independence number of $D_n$ is given by
  $$\alpha(D_n)=\max\limits_{2\leq p\leq n}\left\lfloor\frac{n}{p^{\acute{v}_p(n)}}\right\rfloor,$$
  where $p\in\mathbb{P}$. In particular, if $n$ is odd, we have 
  $$\alpha(D_n)=\left\lfloor\frac{n}{2}\right\rfloor.$$
\end{thm}
\begin{proof}
   Let $f(u),f(v)$ be two labels, where $u,v\in V(D_n)$. Then, using Lemma \ref{lem1},  $uv\notin E(D_n)$ if and only if there exists $p\in\mathbb{P}$ such that $f(u), f(v)\in M_{p^{\acute{v}_{p}(n)}}.$
  Therefore, the set of vertices of $D_n$ with labels in $M_{p^{\acute{v}_p(n)}}$
  is an independent set.
  Hence, 
  $$\alpha(D_n)=\max\limits_{2\leq p\leq n}\left|M_{p^{\acute{v}_p(n)}}\right|=\max\limits_{2\leq p\leq n}\left\lfloor\frac{n}{p^{\acute{v}_p(n)}}\right\rfloor.$$
  If $n$ is an odd, then the set of nonadjacent vertices in $D_n$ with labels in $M_2=\left\{2, 4, 6, \dots, 2\left\lfloor\frac{n}{2}\right\rfloor\right\}$ is a maximal independent set. 
  Hence, $\alpha(D_n)=\left\lfloor\frac{n}{2}\right\rfloor$.
\end{proof}

\begin{cor} \label{corVI4}
  Let $G$ be a graph with order $n$. If the graph $G$ is Diophantine, then $\alpha(G)\geq\alpha(D_n)$.
\end{cor}

\begin{thm}(\textbf{sufficient condition})\label{thm_sufficient}
  Let $G$ be a graph with order $n$. If $\alpha(G)\geq n-F(D_n)$, then $G$ is a Diophantine graph.
\end{thm}
\begin{proof}
  Let $G$ be a graph with $n$ vertices such that $\alpha(G)\geq n-F(D_n)$. Suppose $S$ is a subgraph of $G$ with number of vertices less than or equal to $F(D_n)$ vertices. Then, using Theorem \ref{lem2}, the number that is either a divisor of $n$ or is of the form $p^{\acute{v}_{p}(n)}$, where $\frac{n}{2}<p^{\acute{v}_{p}(n)}<n$ can be used as labels in the vertices of the subgraph $S$ of $G$. Therefore, the other labels can be assigned in the remaining independent vertices of order $\alpha(G)$ from the graph $G$. Hence, $G$ is a Diophantine graph.
\end{proof}

\begin{ex}
  The graph $G=K_3+\overline{K_4}$ satisfies the sufficient condition in Theorem \ref{thm_sufficient} as $\alpha(G)\geq 7-F(D_7)$; accordingly, the graph $G$ is a Diophantine graph as seen in Figure \ref{figure1} part (a).
\begin{figure*}[h!]
\centering
  \begin{subfigure}{0.4\textwidth}
  \centering
  \begin{tikzpicture}
   [scale=.5,auto=center,every node/.style={circle,fill=blue!20}]
  \node (v1)  at (0,3)   {$1$};
  \node (v5)  at (3,4)   {$5$};
  \node (v7)  at (-3,4)  {$7$};
  \node (v3)[circle,fill=red!20] at (3,0)   {$3$};
  \node (v2)[circle,fill=red!20] at (-3,0)  {$2$};
  \node (v6)[circle,fill=red!20] at (1,0)   {$6$};
  \node (v4)[circle,fill=red!20] at (-1,0)  {$4$};
  \draw (v1) -- (v2);
  \draw (v1) -- (v3);
  \draw (v1) -- (v4);
  \draw (v1) -- (v5);
  \draw (v1) -- (v6);
  \draw (v1) -- (v7);

  \draw (v5) -- (v2);
  \draw (v5) -- (v3);
  \draw (v5) -- (v4);
  \draw (v5) -- (v6);
  \draw (v5) -- (v7);

  \draw (v7) -- (v2);
  \draw (v7) -- (v3);
  \draw (v7) -- (v4);
  \draw (v7) -- (v6);
  \end{tikzpicture}\caption{The graph $G=K_3+\overline{K_4}$}  
 \end{subfigure}
 ~~~~~~~
 \begin{subfigure}{0.4\textwidth}
 \centering
  \begin{tikzpicture}
   [scale=.5,auto=center,every node/.style={circle,fill=blue!20}]
  \node (v1) at (1,6)   {$1$};
  \node (v5) at (-1,6)  {$5$};
  \node (v7) at (-3,4)  {$7$};
  \node (v3) at (3,4)   {$3$};
  \node (v2)[circle,fill=red!20] at (-2,2)  {$2$};
  \node (v6)[circle,fill=red!20] at (0,1.5)   {$6$};
  \node (v4)[circle,fill=red!20] at (2,2)   {$4$};
  \draw (v1) -- (v2);
  \draw (v1) -- (v3);
  \draw (v1) -- (v4);
  \draw (v1) -- (v5);
  \draw (v1) -- (v6);
  \draw (v1) -- (v7);

  \draw (v5) -- (v2);
  \draw (v5) -- (v3);
  \draw (v5) -- (v4);
  \draw (v5) -- (v6);
  \draw (v5) -- (v7);

  \draw (v7) -- (v2);
  \draw (v7) -- (v3);
  \draw (v7) -- (v4);
  \draw (v7) -- (v6);

  \draw (v3) -- (v2);
  \draw (v3) -- (v4);
  \end{tikzpicture}\caption{The maximal Diophantine graph $D_7$}
 \end{subfigure}\caption{}\label{figure1} 
\end{figure*}
\end{ex}
This sufficient condition is not a necessary condition. For instance, $D_7$ is a Diophantine graph. However, $D_7$ does not meet the sufficient condition since $\alpha(D_7)<7-F(D_7)$.
\subsection{The Minimum Degree Vertex with Minimum Label}

 \hspace{0.5cm} Let $V_{\delta}(G)$ denote the set of vertices with minimum degree in $G$, i.e.,  $V_{\delta}(G):=\{u\in V(G):\deg(u)=\delta(G)\}$. So, we denote the following  $f(u_0):=\min\{f(u): u\in V_{\delta}(D_n)\},$
  in which the vertex $u_0$ is the vertex in $V_{\delta}(D_n)$ with minimum label and $f(u_0)$ is the smallest label of a vertex in $V_{\delta}(D_n)$.

\begin{rem} Let $u_0$ be the vertex in $V_{\delta}(D_n)$ with minimum label.
  \begin{itemize}
  \item[1.] $f(u_0)=1$ if and only if $D_n$ is a complete graph.
  \item [2.] For every $n>3$, $\delta(D_n)\geq 3$.
   \end{itemize}
\end{rem}

\begin{thm}
  Suppose the maximal Diophantine graph $D_n$ is not a complete graph and the finite sequence 
  $$p_1^{\acute{v}_{p_{_1}}(n)}< p_2^{\acute{v}_{p_{_2}}(n)}<\dots< p_s^{\acute{v}_{p_s}(n)}<\frac{n}{2},$$ 
  where $p_i$, $i = 1, 2, \dots, s$ are distinct prime numbers. If the vertex $u_0\in V_{\delta}(D_n)$ has minimum label, then there exits an integer $r<s$ such that $$f(u_0)=\prod\limits_{i=1}^{r}p_i^{\acute{v}_{p_i}(n)}<n\quad \mbox{and}\quad  p_{r+1}^{\acute{v}_{p_{r+1}}(n)}f(u_0)>n.$$
\end{thm}
\begin{proof}
  Suppose the maximal Diophantine graph $D_n$ of order $n$ is not a complete graph. Let $u_0$ be the vertex in $V_{\delta}(D_n)$ with minimum label. Then
  $$1<f(u_0)=\min\{f(u): u\in V_{\delta}(D_n)\}<n.$$ Therefore, we have 
  \begin{equation}\label{eqn8}
     f(u_0)=\prod\limits_{j=1}^{r_1}p_j^{\alpha_j} \prod\limits_{i=1}^{r_2}q_i^{\beta_i}<n. 
  \end{equation}
  where $p_j,q_i$ are distinct prime numbers and $\alpha_j,\beta_i$ are two non-negative integers  such that for every $j=1,2,\dots r_1$, $i=1,2,\dots r_2$
  \begin{equation}\label{eqn12}
      \alpha_j\geq \acute{v}_{p_j}(n)\quad\mbox{and}\quad 0\leq \beta_i<\acute{v}_{h_i}(n).
  \end{equation}
  Clearly, the two terms of equation \eqref{eqn8} are relatively primes and $\prod\limits_{i=1}^{r_2}q_i^{\beta_i}\mid n$. Otherwise, if $\prod\limits_{i=1}^{r_2}q_i^{\beta_i}\nmid n$, then there exists a prime number $q_i\mid f(u_0)$ such that $q_i^{\beta_i}\nmid n$. Therefore, there exists $i=1,2,\dots r_2$ such that $\beta_i\geq\acute{v}_{q_i}(n)$ which contradicts equation \eqref{eqn12}. Let $u\in V(D_n)$ such that $f(u)=\prod\limits_{j=1}^{r_1}p_j^{\alpha_j}$. Then, using equation \eqref{eqn8}, we get
  \begin{equation}\label{eqn10}
    f(u_0)=f(u)\prod\limits_{i=1}^{r_2}q_i^{\beta_i}. 
  \end{equation}
  Therefore, using Corollary \ref{cor1} and equation \eqref{eqn10},
  $$f(u)\leq f(u_o)\quad \mbox{and}\quad \deg(u)=\deg(u_0),$$
  This contradicts the hypothesise of the minimal label of $u_0$ unless $\beta_i=0$ in equation \eqref{eqn10}. Thus,
  \begin{equation}\label{eqn11}
    f(u_0)=f(u)=\prod\limits_{j=1}^{r_1}p_j^{\alpha_j},
  \end{equation}
  where $\alpha_j\geq \acute{v}_{p_j}(n)$. Then $\alpha_j=\acute{v}_{p_j}(n)+k_j$ for some $k_j\geq 0$. Consequently, using equation \eqref{eqn11}, we get
  \begin{equation}\label{eqn9}
      f(u_0)=\prod\limits_{j=1}^{r_1}p_j^{v_{p_j}(n)+k_j}= \prod\limits_{j=1}^{r_1}p_j^{\acute{v}_{q_j}(n)} \prod\limits_{j=1}^{r_1}p_j^{k_j}.
  \end{equation}
  Let $v\in V(D_n)$ such that $f(v)=\prod\limits_{j=1}^{r_1}p_j^{\acute{v}_{p_j}(n)}.$ Then, using equation \eqref{eqn9}, we obtain
  \begin{equation}\label{eqn1}
    f(u_0)= \prod\limits_{j=1}^{r_1}p_j^{k_j} f(v).
  \end{equation}
  Thus, using Lemma \ref{lem3}, $\deg(v)\leq \deg(u_o)$. Since $u_o$ has minimum degree, we have $\deg(v)=\deg(u_o)$.
  Therefore,
  $$f(v)\leq f(u_o)\quad \mbox{and}\quad \deg(v)=\deg(u_o).$$
  Since $f(u_o)$ is the minimum label and using equation \eqref{eqn1}, we get $k_j=0$. Consequently we have
  \begin{equation}\label{eqn14}
    f(u_0)=f(v)=\prod\limits_{j=1}^{r_1}p_j^{\acute{v}_{p_j}(n)}<n.
  \end{equation}
  Given a finite sequence $1<p_1^{\acute{v}_{p_{_1}}(n)}<p_2^{\acute{v}_{p_{_2}}(n)}<\dots<p_s^{\acute{v}_{p_s}(n)}<\frac{n}{2},$ where $p_i$, $i = 1, 2, \dots, s$ are distinct primes. Then we have
  \begin{equation*}
    \left|M_{p_1^{\acute{v}_{p_{_1}}(n)}}\right| \geq \left|M_{p_2^{\acute{v}_{p_{_2}}(n)}}\right| \geq \dots \geq \left|M_{p_s^{\acute{v}_{p_s}(n)}}\right|.
  \end{equation*}
  Since $p_i^{\acute{v}_{p_i}(n)}\mid f(u_0)$ for some $i=1, 2, \dots s$ and $u_0$ is the vertex in $V_{\delta}(D_n)$ with minimum label, we have the following cases:
  \\
  Case i: If the prime factors of $f(u_0)$ are $i$ primes, then $f(u_0)\in M_{p_1^{\acute{v}_{p_{_1}}(n)}}, \ f(u_0)\in M_{p_2^{\acute{v}_{p_{_2}}(n)}}, \ \dots, \ f(u_0)\in M_{p_i^{\acute{v}_{p_i}(n)}},$
  where $i=1,2,\dots, r$ for some $r<s$. Then, using equation \eqref{eqn14},  we have the following formula 
  $$f(u_0)=\prod\limits_{j=1}^{i}p_j^{\acute{v}_{p_j}(n)}<n.$$
  where $i=1,2,\dots, r$ and $r<s$. Suppose contrarily that for every $i=1,2,\dots, r$ and $r<s$ such that
  $p_{i+1}^{\acute{v}_{p_{i+1}}(n)}f(u_0)<n.$
  Let $w\in V(D_n)$ such that $f(w)=p_{i+1}^{\acute{v}_{p_{i+1}}(n)}f(u_0).$
  Then, using Lemma \ref{lem3}, $\deg(w)\leq \deg(u_o)$. Since the degree of $u_0$ is minimum, we get $\deg(w)=\deg(u_0)$. However, since
  $$f^*(u_0)=\frac{f(u_o)}{(f(u_o),n)}=\prod\limits_{j=1}^{i}p_j\quad\mbox{and}\quad f^*(w)=\frac{f(w)}{(f(w),n)}=\prod\limits_{j=1}^{i+1}p_j,$$
  therefore the reduced labels $f^*(u_0),f^*(u)$ do not have the same prime factors which contradicts Theorem \ref{thm_eq-deq2}. Hence, the proof follows.
\end{proof}

Clearly, one can see that $\delta(D_n)=\deg(u_0)$ and the degree of every vertex in $D_n$ is provided by 
\begin{thm}\cite{Nasr}.\label{deg(u)}
   If $f^*(u)=\prod\limits_{i=1}^{r}p_i^{k_i}$, where $u\in V(D_n)$, $p_i$, $i=1,2,\dots,r$ are distinct prime numbers and $k_i\geq0$, $i=1,2,\dots,r$, then
\begin{equation*}
\deg(u)=\left\{\begin{array}{ll}
               n-1,                                                                                                            & \hbox{$f^*(u)=1.$} \\
               n-\sum\limits_{1\leq i\leq r}\bigg\lfloor\frac{n}{p_i^{\acute{v}_{p_i}(n)}}\bigg\rfloor
                 +\sum\limits_{1\leq i.j\leq r}\bigg\lfloor\frac{n}{p_i^{\acute{v}_{p_i}(n)}p_j^{\acute{v}_{p_j}(n)}}\bigg\rfloor
                 -\dots +(-1)^{r}\Bigg\lfloor\frac{n}{\prod\limits_{1\leq i\leq r}p_i^{\acute{v}_{p_i}(n)}}\Bigg\rfloor,       & \hbox{$f^*(u)>1.$}
              \end{array}
       \right.
\end{equation*}
\end{thm}

\begin{cor}\label{thmVI1}
  Let $G$ be a graph of order $n$. If the graph $G$ is Diophantine, then $\delta(G)\leq \delta(D_n)$.
\end{cor}
\begin{proof}
  Let $G$ be a Diophantine graph  of order $n$. Then, using Corollary \ref{spanning-thm}, the graph $G$ is labeling isomorphic to a spanning subgraph (say $\acute{G}$) of $D_n$, i.e., $G\cong_l\acute{G}$. Hence, $\delta(G)=\delta(\acute{G})\leq\delta(D_n)$
\end{proof}

\begin{thm}\label{lem5}
  There exists a vertex $w\in V_{\delta}(D_n)$ such that $\frac{n}{2}< f(w)< n$.
\end{thm}
\begin{proof}
let $u\in V_{\delta}(D_n)$. Then the two cases follow, in the case of $\frac{n}{2}< f(u)< n$, we have nothing to prove. In the other case of $1<f(u)<\frac{n}{2}$, we get $\frac{n}{2}< 2^t f(u)< n$ for some $t>0$. Let $f(w)=2^tf(u)$, where $w\in V(D_n)$, $t>0$. Then we have $\frac{n}{2}<f(w)< n$. Using Lemma \ref{lem3}, we get $N(u)\supseteq N(w)$. Therefore, $\deg(w)\leq \deg(u)$. Since $u\in V_{\delta}(D_n)$, therefore $\deg(w)=\deg(u)=\delta(D_n)$. Hence, form the two cases, There exists a vertex $w\in V_{\delta}(D_n)$ such that $\frac{n}{2}< f(w)< n$.
\end{proof}
\subsection{The Degree Sequences of graphs}

\begin{dfn}
  Let $G$ be a graph of order $n$. A finite sequence $S_G=\{g_i\}_{i=0}^{n-1}=(g_0, g_1, \dots, g_{n-1})$, where $g_i:=|\{v\in V(G):\deg(v)=i\}|$, $i=0,1, \dots, n-1$ is called the vertex-degree sequence of $G$.
\end{dfn}

The reader notices that the standard notion of degree sequence is different in this literature. Obviously, we obtain the following two equations  $\sum\limits_{i=0}^{n-1}g_i=|V(G)|$ and $\sum\limits_{i=0}^{n-1} ig_i=2|E(G)|$.  

\begin{ex} 
  $S_{D_7}=(0, 0, 0, 1, 2, 1, 3)$, 
  and
  $S_{D_{11}}=(0,0,0,0,1,1,3,0,2,1,3)$. 
\end{ex}


\begin{thm}\label{thmVI2}
  Let $G$ be a graph  of order $n$ with $S_G=\{g_i\}_{i=0}^{n-1}$ and $D_n$ be the maximal Diophantine graph of order $n$ with  $S_{D_n}=\{d_i\}_{i=0}^{n-1}$. If the graph $G$ is Diophantine then for each $k\in\{0,1,\dots,n-1\}$, 
  $$\sum\limits_{i=0}^{k} g_{i}\geq \sum\limits_{i=0}^{k} d_{i}.$$
\end{thm}
\begin{proof}
  Assume $(D_n,f_1)$ is the maximal Diophantine graph of order $n$ with $S_{D_n}=\{d_i\}_{i=0}^{n-1}$ and $(G,f_2)$ is a Diophantine graph  of order $n$ with $S_G=\{g_i\}_{i=0}^{n-1}$, where 
  $$f_1:V(D_n)\rightarrow \{1,2,\dots,n\}, \quad f_2:V(G)\rightarrow \{1,2,\dots,n\}.$$
  are Diophantine labelings of $D_n,G$ respectively. Suppose contrarily that there exists $k_0 \in\{0,1,\dots,n-1\}$ such that
  \begin{equation}\label{eqn5}
      \sum\limits_{i=0}^{k_0} g_{i}< \sum\limits_{i=0}^{k_0} d_{i}.
  \end{equation}
  Using Corollary \ref{spanning-thm}, we get that $(G,f_2)$ is labeling isomorphic to a spanning subgraph of $(D_n,f_1)$. Let $(\acute{G},f_1)$ be a spanning subgraph of $(D_n,f_1)$ such that 
  $(G,f_2)\cong_l(\acute{G},f_1).$ Then there is a bijective map $\varphi:V(G)\rightarrow V(\acute{G})$ such that for all $u,\acute{u}\in V(G)$, 
  $u\acute{u}\in E(G)$  if and only if $\varphi(u)\varphi(\acute{u})\in E(\acute{G})$ and $f_1(u)=f_2(\varphi(u)).$
  Then for every $u\in V(G)$,
  \begin{equation}\label{eqn2}
      \deg(u)=\deg(\varphi(u)).
  \end{equation}
  Define a map $\acute{\varphi}:V(G)\rightarrow V(D_n)$ such that for every $u\in V(G)$, $\acute{\varphi}(u):=\varphi(u).$ 
  Then the map $\acute{\varphi}$ is bijective and for every $u\in V(G)$, $f_2(u)=f_1(\varphi(u))=f_1(\acute{\varphi}(u))$.
  Since $(\acute{G},f_1)$ is a spanning subgraph of $(D_n,f_1)$ and using equation \eqref{eqn2}, one can see that for every $u\in V(G)$,
  \begin{equation}\label{eqn3}
      \deg(u)\leq \deg(\acute{\varphi}(u)),
  \end{equation} 
 Define the following two sets 
 $$A_{D_n}(k):=\{v\in V(D_n):\deg(v)\leq k\}\quad\mbox{and} \quad A_G(k):=\{u\in V(G):\deg(u)\leq k\},$$ 
 where $k=0,1,\dots,n-1$. Therefore, using equation \eqref{eqn5}, there exists $k_0\in\{ 0,1,\dots,n-1\}$ such that
 \begin{equation*}
    |A_{G}(k_0)|=\sum\limits_{i=0}^{k_0} g_{i}< \sum\limits_{i=0}^{k_0} d_{i}=|A_{D_n}(k_0)|.
 \end{equation*}
  Consequently, there exists a vertex $v\in A_{D_n}(k_0)\subseteq V(D_n)$ such that $u=\acute{\varphi}^{-1}(v)\notin A_{G}(k_0).$ Then,
  \begin{equation*}
      \deg(u)> k_0\geq\deg(v)=\deg(\acute{\varphi}(u)),
  \end{equation*}
  which contradicts equation \eqref{eqn3}. Hence, for every $k\in \{0,1,\dots,n-1\}$,
  \begin{equation*}
      \sum\limits_{i=0}^{k} g_{i}\geq  \sum\limits_{i=0}^{k} d_{i},
  \end{equation*}
  which completes the proof
\end{proof}

\begin{rem}\label{cor2}
  Let a graph $G$ of order $n$ have $S_G=\{g_i\}_{i=0}^{n-1}$ and the maximal Diophantine graph $D_n$ have  $S_{D_n}=\{d_i\}_{i=0}^{n-1}$.  If $\sum\limits_{i=0}^{k} g_{i}\geq\sum\limits_{i=0}^{k} d_{i}$ holds for every $k=0,1,\dots,n-1$, then 
  $$|E(G)|\leq|E(D_n)|\quad \mbox{or}\quad F(G)\leq F(D_n)\quad\mbox{or}\quad \delta(G)\leq\delta(D_n).$$
\end{rem}

The proof of Remark \ref{cor2} closely resembles the corollaries in \cite{Seoud1}. The following table presents the quantities  $|E(G)|$, $F(G)$, $Cl(G)$, $\alpha(G)$, $\delta(G)$ and $S_{G}$, where $G=D_n$, $n=4,\dots,20$. 

\begin{table}[h]
\centering
\begin{tabular}{c c c c c c c c}
\hline
$n$  & $|E(D_n)|$  & $F(D_n)$  & $Cl(D_n)$ & $\alpha(D_n)$ & $\delta(D_n)$  & $S_{D_n}=(d_i)$                            \\
\hline\hline
4  & 6           & 4           & 4           & 1             & 3              & $(0,0,0,4)$                                  \\
5  & 9           & 3           & 4           & 2             & 3              & $(0,0,0,2,3)$                                \\
6  & 15          & 6           & 6           & 1             & 5              & $(0,0,0,0,0,6)$                              \\
7  & 17          & 3           & 5           & 3             & 3              & $(0,0,0,1,2,1,3)$                            \\
8  & 27          & 6           & 7           & 2             & 6              & $(0,0,0,0,0,0,2,6)$                          \\
9  & 30          & 5           & 6           & 4             & 5              & $(0,0,0,0,0,4,0,0,5)$                        \\
10 & 41          & 5           & 7           & 3             & 7              & $(0,0,0,0,0,0,0,3,2,5)$                      \\
11 & 41          & 3           & 6           & 5             & 4              & $(0,0,0,0,1,1,3,0,2,1,3)$                    \\
12 & 65          & 10          & 11          & 2             & 10             & $(0,0,0,0,0,0,0,0,0,0,2,10)$                 \\
13 & 57          & 4           & 7           & 6             & 5              & $(0,0,0,0,0,2,1,3,0,2,0,1,4)$                \\
14 & 81          & 6           & 8           & 4             & 8              & $(0,0,0,0,0,0,0,0,1,0,3,2,2,6)$              \\
15 & 83          & 7           & 9           & 7             & 7              & $(0,0,0,0,0,0,0,1,6,0,0,0,0,1,7)$            \\
16 & 106         & 7           & 10          & 5             & 9              & $(0,0,0,0,0,0,0,0,0,1,0,4,0,2,2,7)$          \\
17 & 95          & 4           & 8           & 8             & 6              & $(0,0,0,0,0,0,2,1,1,4,1,0,2,0,1,1,4)$        \\
18 & 143         & 9           & 11          & 4             & 14             & $(0,0,0,0,0,0,0,0,0,0,0,0,0,0,4,3,2,9)$      \\
19 & 119         & 5           & 9           & 9             & 7              & $(0,0,0,0,0,0,0,3,1,1,4,1,0,2,0,0,1,1,5)$    \\
20 & 173         & 10          & 13          & 6             & 14             & $(0,0,0,0,0,0,0,0,0,0,0,0,0,0,6,0,0,0,4,10)$ \\
\hline
\end{tabular}
\caption{Basic Bounds of $D_n$, $n=4,\dots,20$}
\label{table1}
\end{table}
\section{Necessary Conditions for Diophantine Graphs}
  \hspace{0.5cm} According to Definition \ref{dfn2}, Corollaries \ref{corVI2}, \ref{corVI3}, \ref{corVI4}, \ref{thmVI1} and Theorem  \ref{thmVI2},  each of the following six conditions $\textbf{C1}, \dots,\textbf{C6}$ listed below constitutes a necessary condition for the existence of a Diophantine labeling for a given graph $G$ of order $n$:\\

  \textbf{C1}. $|E(G)|\leq|E(D_n)|$.\quad \textbf{C2}. $F(G)\leq F(D_n)$.$\quad$ \textbf{C3}. $Cl(G)\leq Cl(D_n)$.$\quad$ \textbf{C4}. $\alpha(G)\geq\alpha(D_n)$.\quad

  \textbf{C5}. $\delta(G)\leq\delta(D_n)$.\hspace{0.8cm} \textbf{C6}. There exists $k\in\{0, 1, \dots, n-1\}$ such that $\sum\limits_{i=0}^{k} g_{i}\geq\sum\limits_{i=0}^{k} d_{i}$.\\

Notice that, based on Remark \ref{cor2} and Example \ref{example1}, \textbf{C6} is stronger than each of the following three conditions \textbf{Ci}, i = 1, 2, 5. Conditions \textbf{Ci}, i = 1, 2, 3, 4, 5 are mutually independent while conditions \textbf{Ci}, i = 3, 4, 6 are also mutually independent. Additionally, Examples \ref{example1} and \ref{example2} illustrate some relations among these conditions.

\begin{ex}\label{example1}
  In this example, the previous six necessary conditions are studied in the following six graphs $G_i$, $i=1,\dots,6$ as shown in Figure \ref{figure2}. The basic bounds for these graphs are given in Table \ref{table2} and the corresponding maximal Diophantine graphs $D_7$ and $D_{11}$. The graph $G_1$ does not satisfy \textbf{C4}, the graph $G_2=C_4+\overline{K_3}$ does not satisfy \textbf{C5}, \textbf{C6}   $\left(\mbox{for} \ \sum\limits_{i=0}^{3}g_{i}<\sum\limits_{i=0}^{3}d_{i}\right)$, the graph $G_3$ does not satisfy \textbf{C3}, the graph $G_4=K_4+\overline{K_7}$ does not satisfy \textbf{C2}, \textbf{C6} $\left(\mbox{for} \ \sum\limits_{i=0}^{9} g_{i}<\sum\limits_{i=0}^{9} d_{i} \right)$, the graph $G_5$  does not satisfy \textbf{C1}, \textbf{C6} $\left(\mbox{for} \ \sum\limits_{i=0}^{5} g_{i}<\sum\limits_{i=0}^{5} d_{i} \right)$ and the graph $G_6$ does not satisfy \textbf{C6} $\left(\mbox{for} \ \sum\limits_{i=0}^{8}g_{i}<\sum\limits_{i=0}^{8}d_{i} \right)$. Therefore, the graphs $G_i$, $i=1,\dots,6$ are not Diophantine. However, the other necessary conditions are satisfied for the graphs $G_i$, $i=1,2,\dots,6$.
\end{ex}

\begin{figure*}[h!]
\centering
 \begin{subfigure}{0.3\textwidth}
 \centering
   \begin{tikzpicture}
[scale=.7,auto=center,every node/.style={circle,fill=blue!20}]
  \node (v1)[circle,fill=red!20] at (2,0)       {$4$};
  \node (v2)  at (0,2)       {$2$};
  \node (v3)  at (-2,0)      {$3$};
  \node (v4)  at (0,-2)      {$1$};
  \node (v5)  at (2,-2)      {$5$};
  \node (v6)[circle,fill=red!20] at (4,0)       {$6$};
  \node (v7)  at (2,2)       {$7$};
\draw (v1) -- (v2);
\draw (v1) -- (v3);
\draw (v1) -- (v4);
\draw (v1) -- (v5);
\draw (v1) -- (v7);

\draw (v2) -- (v3);
\draw (v2) -- (v4);
\draw (v2) -- (v5);
\draw (v2) -- (v7);

\draw (v3) -- (v4);
\draw (v3) -- (v5);
\draw (v3) -- (v7);

\draw (v4) -- (v5);

\draw (v5) -- (v6);

\draw (v6) -- (v7);
   \end{tikzpicture}\caption{Graph $G_1$}
  \end{subfigure}
~
 \begin{subfigure}{0.3\textwidth}
 \centering
  \begin{tikzpicture}
[scale=.7,auto=center,every node/.style={circle,fill=blue!20}]
  \node (v1)  at (1.8,2)        {$1$};
  \node (v2)  at (1.8,0.6)      {$2$};
  \node (v3)  at (0,-2)       {$3$};
  \node (v4)  at (0,-0.6)     {$4$};

  \node (v5)[circle,fill=red!20] at (6,1.5)        {$5$};
  \node (v6)[circle,fill=red!20] at (6,0)          {$6$};
  \node (v7)[circle,fill=red!20] at (6,-1.5)       {$7$};
\draw (v1) -- (v2);
\draw (v1) -- (v4);
\draw (v1) -- (v5);
\draw (v1) -- (v6);
\draw (v1) -- (v7);

\draw (v2) -- (v3);
\draw (v2) -- (v5);
\draw (v2) -- (v6);
\draw (v2) -- (v7);

\draw (v3) -- (v4);
\draw (v3) -- (v5);
\draw (v3) -- (v6);
\draw (v3) -- (v7);

\draw (v4) -- (v5);
\draw (v4) -- (v6);
\draw (v4) -- (v7);
  \end{tikzpicture}\caption{Graph $G_2$}
 \end{subfigure}
~
 \begin{subfigure}{0.3\textwidth}
 \centering
  \begin{tikzpicture}
[scale=.6,auto=center,every node/.style={circle,fill=blue!20}]
  \node (v1)[circle,fill=red!20] at (2,2)          {$1$};
  \node (v2)[circle,fill=red!20] at (2,0.8)        {$2$};
  \node (v3)[circle,fill=red!20] at (2,-0.8)       {$3$};
  \node (v4)[circle,fill=red!20] at (2,-2)         {$4$};

  \node (v5)   at (-2,-2)    {$6$};
  \node (v6)   at (-1,0)     {$7$};
  \node (v7)   at (-2,2)     {$8$};
  \node (v8)[circle,fill=red!20] at (-4,2.5)     {$9$};
  \node (v9)   at (-6,1.3)     {$10$};
  \node (v10)  at (-6,-1.3)    {$11$};
  \node (v11)  at (-4,-2.5)    {$5$};
\draw (v1) -- (v5);
\draw (v1) -- (v6);
\draw (v1) -- (v7);

\draw (v2) -- (v5);
\draw (v2) -- (v6);
\draw (v2) -- (v7);

\draw (v3) -- (v5);
\draw (v3) -- (v6);
\draw (v3) -- (v7);

\draw (v4) -- (v5);
\draw (v4) -- (v6);
\draw (v4) -- (v7);

\draw (v5) -- (v6);
\draw (v5) -- (v7);
\draw (v5) -- (v8);
\draw (v5) -- (v9);
\draw (v5) -- (v10);
\draw (v5) -- (v11);

\draw (v6) -- (v7);
\draw (v6) -- (v8);
\draw (v6) -- (v9);
\draw (v6) -- (v10);
\draw (v6) -- (v11);

\draw (v7) -- (v8);
\draw (v7) -- (v9);
\draw (v7) -- (v10);
\draw (v7) -- (v11);

\draw (v8) -- (v9);
\draw (v8) -- (v10);
\draw (v8) -- (v11);

\draw (v9) -- (v10);
\draw (v9) -- (v11);

\draw (v10) -- (v11);
  \end{tikzpicture}\caption{Graph $G_3$}
 \end{subfigure}
~
 \begin{subfigure}{0.3\textwidth}
 \centering
  \begin{tikzpicture}
[scale=.6,auto=center,every node/.style={circle,fill=blue!20}]
  \node (v1)   at (3,1)        {$1$};
  \node (v2)   at (1,3)        {$2$};
  \node (v3)   at (-1,3)       {$3$};
  \node (v4)   at (-3,1)       {$4$};

  \node (v5)[circle,fill=red!20]  at (-3.5,-1)    {$10$};
  \node (v6)[circle,fill=red!20]  at (-2.5,-2)    {$5$};
  \node (v7)[circle,fill=red!20]  at (-1.3,-2.5)  {$6$};
  \node (v8)[circle,fill=red!20]  at (0,-3)       {$7$};
  \node (v9)[circle,fill=red!20]  at (1.3,-2.5)   {$8$};
  \node (v10)[circle,fill=red!20] at (2.5,-2)     {$9$};
  \node (v11)[circle,fill=red!20] at (3.5,-1)     {$11$};
\draw (v1) -- (v2);
\draw (v1) -- (v3);
\draw (v1) -- (v4);
\draw (v1) -- (v5);
\draw (v1) -- (v6);
\draw (v1) -- (v7);
\draw (v1) -- (v8);
\draw (v1) -- (v9);
\draw (v1) -- (v10);
\draw (v1) -- (v11);

\draw (v2) -- (v3);
\draw (v2) -- (v4);
\draw (v2) -- (v5);
\draw (v2) -- (v6);
\draw (v2) -- (v7);
\draw (v2) -- (v8);
\draw (v2) -- (v9);
\draw (v2) -- (v10);
\draw (v2) -- (v11);

\draw (v3) -- (v4);
\draw (v3) -- (v5);
\draw (v3) -- (v6);
\draw (v3) -- (v7);
\draw (v3) -- (v8);
\draw (v3) -- (v9);
\draw (v3) -- (v10);
\draw (v3) -- (v11);

\draw (v4) -- (v5);
\draw (v4) -- (v6);
\draw (v4) -- (v7);
\draw (v4) -- (v8);
\draw (v4) -- (v9);
\draw (v4) -- (v10);
\draw (v4) -- (v11);
  \end{tikzpicture}\caption{Graph $G_4$}
 \end{subfigure}
~
 \begin{subfigure}{0.3\textwidth}
 \centering
  \begin{tikzpicture}
[scale=.6,auto=center,every node/.style={circle,fill=blue!20}]
  \node (v1)   at (3.7,0)        {$1$};
  \node (v2)   at (3,2.5)        {$2$};
  \node (v3)   at (1.5,4)        {$3$};
  \node (v4)   at (-1.5,4)       {$4$};
  \node (v5)   at (-3,2.5)       {$5$};
  \node (v6)   at (-3.7,0)       {$6$};

  \node (v7)[circle,fill=red!20]  at (-3,-2.5)    {$10$};
  \node (v8)[circle,fill=red!20]  at (-1.5,-3)    {$7$};
  \node (v9)[circle,fill=red!20]  at (0,-3)       {$8$};
  \node (v10)[circle,fill=red!20] at (1.5,-3)     {$9$};
  \node (v11)[circle,fill=red!20] at (3,-2.5)     {$11$};
\draw (v1) -- (v2);
\draw (v1) -- (v3);
\draw (v1) -- (v4);
\draw (v1) -- (v5);
\draw (v1) -- (v6);
\draw (v1) -- (v7);
\draw (v1) -- (v8);
\draw (v1) -- (v9);
\draw (v1) -- (v10);
\draw (v1) -- (v11);

\draw (v2) -- (v3);
\draw (v2) -- (v4);
\draw (v2) -- (v5);
\draw (v2) -- (v7);
\draw (v2) -- (v8);
\draw (v2) -- (v9);
\draw (v2) -- (v10);
\draw (v2) -- (v11);

\draw (v3) -- (v4);
\draw (v3) -- (v5);
\draw (v3) -- (v6);
\draw (v3) -- (v7);
\draw (v3) -- (v8);
\draw (v3) -- (v9);
\draw (v3) -- (v10);
\draw (v3) -- (v11);

\draw (v4) -- (v5);
\draw (v4) -- (v6);
\draw (v4) -- (v8);
\draw (v4) -- (v9);
\draw (v4) -- (v10);
\draw (v4) -- (v11);

\draw (v5) -- (v6);
\draw (v5) -- (v8);
\draw (v5) -- (v9);
\draw (v5) -- (v10);
\draw (v5) -- (v11);

\draw (v6) -- (v7);
\draw (v6) -- (v8);
\draw (v6) -- (v9);
\draw (v6) -- (v10);
\draw (v6) -- (v11);
  \end{tikzpicture}\caption{Graph $G_5$}
 \end{subfigure}
~
 \begin{subfigure}{0.3\textwidth}
 \centering
  \begin{tikzpicture}
[scale=.6,auto=center,every node/.style={circle,fill=blue!20}]
  \node[circle,fill=blue!20](v1)   at (4,1)      {$1$};
  \node[circle,fill=blue!20](v2)   at (2.5,3)    {$2$};
  \node[circle,fill=blue!20](v3)   at (0,4)      {$3$};
  \node[circle,fill=blue!20](v4)   at (-2.5,3)   {$4$};
  \node[circle,fill=blue!20](v5)   at (-4,1)     {$5$};

  \node[circle,fill=red!20](v6)   at (-4,-1)     {$6$};
  \node[circle,fill=red!20](v7)   at (-2.5,-1.5) {$7$};
  \node[circle,fill=red!20](v8)   at (0,-2)      {$8$};
  \node[circle,fill=red!20](v9)   at (2.5,-1.5)  {$9$};
  \node[circle,fill=red!20](v10)  at (4,-1)      {$10$};

  \node[circle,fill=red!20](v11)  at (0,5.5)     {$11$};

\draw (v2) -- (v1);
\draw (v2) -- (v3);
\draw (v2) -- (v4);
\draw (v2) -- (v5);
\draw (v2) -- (v6);
\draw (v2) -- (v7);
\draw (v2) -- (v8);
\draw (v2) -- (v9);
\draw (v2) -- (v10);
\draw (v2) -- (v11);

\draw (v3) -- (v1);
\draw (v3) -- (v4);
\draw (v3) -- (v5);
\draw (v3) -- (v6);
\draw (v3) -- (v7);
\draw (v3) -- (v8);
\draw (v3) -- (v9);
\draw (v3) -- (v10);
\draw (v3) -- (v11);

\draw (v4) -- (v1);
\draw (v4) -- (v5);
\draw (v4) -- (v6);
\draw (v4) -- (v7);
\draw (v4) -- (v8);
\draw (v4) -- (v9);
\draw (v4) -- (v10);
\draw (v4) -- (v11);

\draw (v1) -- (v5);
\draw (v1) -- (v6);
\draw (v1) -- (v7);
\draw (v1) -- (v8);
\draw (v1) -- (v9);
\draw (v1) -- (v10);

\draw (v5) -- (v6);
\draw (v5) -- (v7);
\draw (v5) -- (v8);
\draw (v5) -- (v9);
\draw (v5) -- (v10);
  \end{tikzpicture}\caption{Graph $G_6$}
 \end{subfigure}\caption{The graphs $G_1,G_2,G_3,G_4,G_5$ and $G_6$ are non-Diophantine}\label{figure2}
\end{figure*}
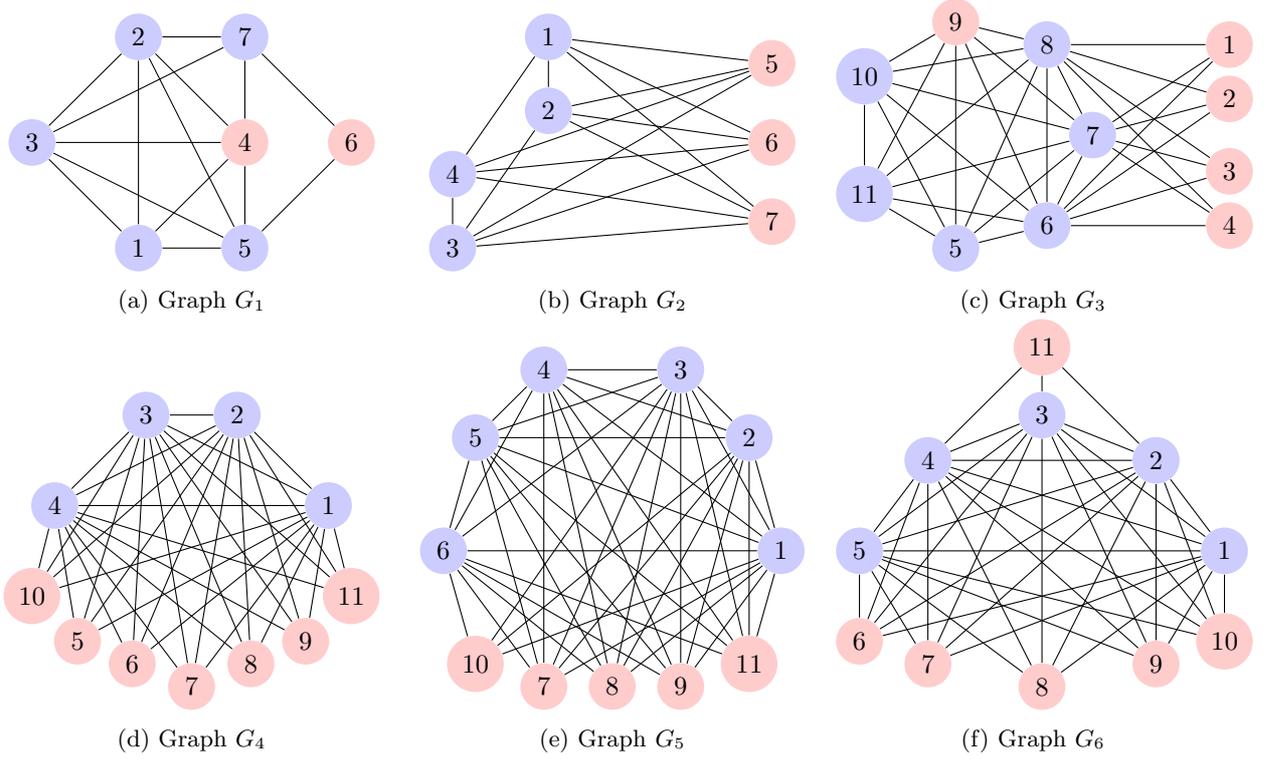

\begin{table}[h!]
  \centering
  \begin{tabular}{|c|c|c|c|c|c|c|}
  \hline
  Graph $G$  & $|E(G)|$      &   $F(G)$    &   $Cl(G)$ & $\alpha(G)$ & $\delta(G)$  & $S_{G}$    \\
  \hline
    $D_7$    &     17        &    3        &     5        &    3        &     3        & $(0,0,0,1,2,1,3)$    \\
  \hline
    $G_1$    &     15        &    0        &     5        &    \fbox{2} &     2        & $(0,0,1,0,2,4,0)$    \\
    $G_2$    &     16        &    0        &     3        &    3        &     \fbox{4} & $\fbox{(0,0,0,0,3,4,0)}$    \\
  \hline
    $D_{11}$ &     41        &    3        &     6        &    5        &     4        & $(0,0,0,0,1,1,3,0,2,1,3)$  \\
  \hline
    $G_3$    &     33        &    3        &     \fbox{7} &    5        &     3        & $(0,0,0,4,0,0,4,0,0,0,3)$  \\
    $G_4$    &     34        &    \fbox{4} &     5        &    7        &     4        & $\fbox{(0,0,0,0,7,0,0,0,0,0,4)}$  \\
    $G_5$    &     \fbox{42} &    2        &     6        &    5        &     5        & $\fbox{(0,0,0,0,1,0,4,0,2,2,2)}$  \\
    $G_6$    &     38        &    3        &     6        &    6        &     3        & $\fbox{(0,0,0,1,0,5,0,0,0,2,3)}$  \\
  \hline
  \end{tabular}
  \caption{Some Basic Bounds of $G_i$, $i=1,\dots,6$ and The Corresponding Bounds of $D_7$ and $D_{11}$}\label{table2}
\end{table}

\begin{ex}\label{example2}
  The following graph $G$ in Figure \ref{figure3} is not Diophantine (see \cite{sonbaty}). However, condition \textbf{C1} is satisfied for $|E(G)|=37<41=|E(D_{11})|$, condition \textbf{C2} is satisfied for $F(G)=3=F(D_{11})$, condition \textbf{C3} is satisfied for $Cl(G)=6=Cl(D_{11})$, condition \textbf{C4} is satisfied for $\alpha(G)=6>5=\alpha(D_{11})$, condition \textbf{C5} is satisfied for $\delta(G)=3<4=\delta(D_{11})$ and condition \textbf{C6} is satisfied for the two sequences 
  $$S_G=\{g_i\}_{i=0}^{10}=(0,0,0,1,1,4,0,0,1,1,3)\quad\mbox{and}\quad S_{D_{11}}=\{d_i\}_{i=0}^{10}=(0,0,0,0,1,1,3,0,2,1,3)$$
  hold the following condition, for every $k=0,1,\dots,10$
  $$\sum\limits_{i=0}^{k}g_{i}\geq\sum\limits_{i=0}^{k}d_{i}.$$ 
  Hence, all six conditions \textbf{Ci}, i=1, \dots, 6 are necessary and insufficient conditions.
  \end{ex}
  
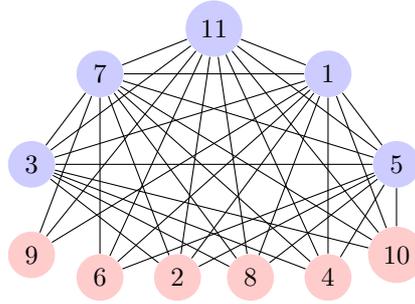
\begin{figure*}[h!]
 \centering
   \begin{tikzpicture}
   [scale=0.6,auto=center]
   \node[circle,fill=blue!20](v1)   at (4,1)      {$5$};
   \node[circle,fill=blue!20](v2)   at (2.5,3)    {$1$};
   \node[circle,fill=blue!20](v3)   at (0,4)      {$11$};
   \node[circle,fill=blue!20](v4)   at (-2.5,3)   {$7$};
   \node[circle,fill=blue!20](v5)   at (-4,1)     {$3$};

   \node[circle,fill=red!20](v8)   at (-0.8,-1.5)     {$2$};
   \node[circle,fill=red!20](v6)   at (-2.5,-1.5)     {$6$};
   \node[circle,fill=red!20](v9)   at (0.8,-1.5)      {$8$};
   \node[circle,fill=red!20](v7)   at (2.5,-1.5)      {$4$};
   \node[circle,fill=red!20](v11)  at (-4,-1)         {$9$};
   \node[circle,fill=red!20](v10)  at (4,-1)          {$10$};
   \draw (v2) -- (v1);
   \draw (v2) -- (v3);
   \draw (v2) -- (v4);
   \draw (v2) -- (v5);
   \draw (v2) -- (v6);
   \draw (v2) -- (v7);
   \draw (v2) -- (v8);
   \draw (v2) -- (v9);
   \draw (v2) -- (v10);
   \draw (v2) -- (v11);

   \draw (v3) -- (v1);
   \draw (v3) -- (v4);
   \draw (v3) -- (v5);
   \draw (v3) -- (v6);
   \draw (v3) -- (v7);
   \draw (v3) -- (v8);
   \draw (v3) -- (v9);
   \draw (v3) -- (v10);
   \draw (v3) -- (v11);

   \draw (v4) -- (v1);
   \draw (v4) -- (v5);
   \draw (v4) -- (v6);
   \draw (v4) -- (v7);
   \draw (v4) -- (v8);
   \draw (v4) -- (v9);
   \draw (v4) -- (v10);
   \draw (v4) -- (v11);

   \draw (v1) -- (v5);
   \draw (v1) -- (v6);
   \draw (v1) -- (v7);
   \draw (v1) -- (v8);
   \draw (v1) -- (v9);
   \draw (v1) -- (v10);

   \draw (v5) -- (v7);
   \draw (v5) -- (v8);
   \draw (v5) -- (v9);
   \draw (v5) -- (v10);
   \end{tikzpicture}
\caption{A graph $G$ does not satisfy the six necessary conditions thought $G$ is not Diophantine}
\label{figure3}
\end{figure*}
\newpage
\begin{flushleft}
\Large\textbf{Declarations}
\end{flushleft}
\begin{flushleft}
\textbf{Conflict of interest}: The authors declare that they have no conflict of interests.
\end{flushleft}



\end{document}